\documentclass[11pt]{article}
\usepackage{amsmath, amssymb, latexsym, amscd, amsthm, amsfonts, amstext}
\usepackage[mathscr]{eucal}
\usepackage[colorlinks]{hyperref}
\usepackage{xspace}
\usepackage{amscd}
\usepackage{amsmath}
\usepackage{amssymb}
\usepackage{latexsym}
\usepackage{graphics}
\usepackage{graphicx}
\usepackage{tikz}
\usepackage{xcolor}
\usepackage{float}
\usepackage{booktabs}

\newtheorem{theorem}{Theorem}[section]

\newtheorem{lemma}[theorem]{Lemma}
\newtheorem{remark}[theorem]{Remark}
\newtheorem{proposition}[theorem]{Proposition}

\newtheorem{example}[theorem]{Example}
\theoremstyle{definition}
\newtheorem{definition}[theorem]{Definition}
\begin{document}

\title{\textbf{Fuzzy subhyperspaces generated\\ by admissible mappings}}

\author {\bf O. R. Dehghan$^\dag$, R. Ameri$^\ddag$\\
        \small \em $^\dag$Department of Mathematics, Faculty of Basic Sciences,\\
        \small \em University of Bojnord, Bojnord, Iran, dehghan@ub.ac.ir\\
        \small \em $^\ddag$School of Mathematics, Statistics and Computer Sciences, College of Sciences,\\
        \small \em University of Tehran, P.O. Box 14155-6455, Tehran, Iran, rameri@ut.ac.ir}

\date{}
\maketitle
\begin{abstract}
Fuzzy hypervector spaces provide a structural framework for modeling graded uncertainty in algebraic systems with multivalued operations. In this paper, we introduce and investigate admissible mappings as a constructive mechanism for generating fuzzy subhyperspaces. These mappings represent fuzzy substructures through families of fuzzy points and allow a systematic characterization of generators. We prove that maximal admissible mappings generate fuzzy subhyperspaces in a manner analogous to bases in classical linear theory, and we present examples over real and finite hypervector spaces to clarify the role of maximality. The proposed approach contributes to the structural foundations of fuzzy modeling in hyperalgebraic settings and may support further developments in soft computing environments involving graded and non-deterministic structures.
\end{abstract}
\section{Introduction}
The theory of algebraic hyperstructures emerged as a natural generalization of classical algebraic structures through the introduction of hyperoperations, where the result of an operation is a set rather than a single element. This concept was first introduced by Marty~\cite{Marty} in 1934 and led to the systematic development of hypergroups, hyperrings, and related structures. Since then, the theory has been extensively developed, and its foundations and applications have been thoroughly studied in several monographs and survey works, notably those by Corsini-Leoreanu~\cite{Corsini Leoreanu}, Davvaz~\cite{Davvaz book 4}, and Davvaz–Leoreanu~\cite{Davvaz book 1}, which established hyperstructure theory as an important branch of modern algebra.

Among algebraic hyperstructures, hypervector spaces play a central role as multivalued generalizations of classical vector spaces. The notion of a hypervector space was introduced by Tallini~\cite{Tallini 1} in the early 1990s and subsequently investigated by several authors, including Ameri~\cite{Ameri Dehghan 1}, Raja~\cite{Raja}, Sedghi~\cite{Sedghi Dehghan Norouzi}, and the first author~\cite{Dehghan Soft HVS}. Fundamental linear concepts such as linear independence, spanning sets, bases, and dimension have been redefined and studied in this framework, showing that many classical linear phenomena persist, albeit in a more subtle form, under multivalued scalar multiplication.

The concept of a fuzzy subset was introduced by Zadeh~\cite{Zadeh} in 1965 as a powerful tool for modeling uncertainty and vagueness in mathematical systems. In contrast to classical subsets, elements of a fuzzy set belong to the set with varying degrees of membership. This idea has had a profound impact on both pure and applied mathematics and has provided a flexible framework for extending classical theories to situations involving incomplete or imprecise information.

Following Zadeh’s pioneering work, Rosenfeld~\cite{Rosenfeld} extended fuzzy set theory to group theory, thereby initiating the study of fuzzy algebraic structures. Since then, fuzzy groups, fuzzy rings, and fuzzy vector spaces have been widely investigated by many researchers, including Katsaras~\cite{Katsaras}, Nanda~\cite{Nanda}, Malik~\cite{Malik Mordeson}, and Kumar~\cite{Kumar}, while more recent contributions have addressed structural characterizations via fuzzy substructures \cite{Mahboob} and the development of fuzzy inner product spaces within linear frameworks \cite{Xiao}. These studies demonstrated that combining algebraic structures with fuzziness yields rich mathematical theories capable of capturing uncertainty within algebraic operations and relations.

Given the success of fuzzy algebra in classical settings, it was natural to extend fuzzy concepts to algebraic hyperstructures. Fuzzy algebraic hyperstructures incorporate both multivalued operations and graded membership, thus simultaneously addressing uncertainty and non-determinism. In recent years, this area has attracted increasing attention, and various properties of fuzzy hyperstructures, their substructures, and related algebraic behaviors have been systematically developed in the literature (see, for instance, \cite{Davvaz book 3}). More recently, fuzzy and soft extensions of algebraic hyperstructures have also been investigated, including studies on soft hyperrings and related regularity notions \cite{Ostadhadi}, as well as state-theoretic investigations in bounded hyper EQ-algebras \cite{Xin}.

Fuzzy hypervector spaces, or equivalently fuzzy subhyperspaces, were first introduced by Ameri~\cite{Ameri FHVS over VF} and later further developed by himself~\cite{Ameri Dehghan 3, Ameri Dehghan 4} and the first author through a series of works \cite{Dehghan Balanced Absorbing, Dehghan Bipolar3, Dehghan Bipolar2}. These structures generalize fuzzy vector spaces by allowing scalar multiplication to be multivalued. The study of fuzzy hypervector spaces raises new challenges, particularly with respect to linear independence, bases, generators, and dimension, which require approaches different from those used in classical fuzzy vector spaces.

The main purpose of this paper is to investigate fuzzy subhyperspaces through the notion of admissible mappings. After recalling the necessary preliminaries, we introduce admissible mappings associated with a fuzzy subhyperspace and study their fundamental properties. We show that maximal admissible mappings play a role analogous to bases in classical linear theory, in the sense that they generate the underlying fuzzy subhyperspace. Several illustrative examples over real and finite hypervector spaces are provided to clarify the theory and highlight the importance of maximality. The results presented here offer an effective algebraic framework for constructing and analyzing generators of fuzzy subhyperspaces.

Beyond its purely algebraic interest, the study of fuzzy subhyperspaces is relevant to the structural modeling of graded and non-deterministic systems, which frequently arise in soft computing contexts. Hyperoperations naturally capture non-uniqueness, while fuzzy membership encodes graded uncertainty. The admissible mapping framework introduced in this paper offers a constructive representation mechanism for such structures, thereby enriching the theoretical foundations underlying fuzzy modeling in multivalued environments.
\section{Preliminaries}
In this section, we recall some basic definitions and properties of hypervector spaces that will be used throughout the paper.
\begin{definition}\label{D HVS}\cite{Tallini 1}
Let $K$ be a field, $(V,+)$ an Abelian group, and $P_{\ast}(V)$ the set of all nonempty subsets of $V$. A \emph{hypervector space} over $K$ is a quadruple $(V,+,\circ,K)$, where ``$\circ$'' is an external hyperoperation
\begin{equation*}
\circ : K\times V \longrightarrow P_{\ast }(V),
\end{equation*}
such that for all $a,b\in K$ and $x,y\in V$, the following conditions hold:
\begin{enumerate}
\item[1)] $a\circ (x+y)\subseteq a\circ x+a\circ y$, \hfill (right distributive law),
\item[2)] $(a+b)\circ x\subseteq a\circ x+b\circ x$, \hfill (left distributive law),
\item[3)] $a\circ (b\circ x)=(ab)\circ x$,
\item[4)] $a\circ (-x)=(-a)\circ x=-(a\circ x)$,
\item[5)] $x\in 1\circ x$,
\end{enumerate}
where in (1), $a\circ x+a\circ y=\{p+q : p\in a\circ x,\ q\in a\circ y\}$, and similarly for (2). Moreover, in (3), $a\circ (b\circ x)=\bigcup_{t\in b\circ x}a\circ t$.

A hypervector space $V$ is called \emph{strongly right distributive} if equality holds in (1). Similarly, one defines \emph{strongly left distributive} hypervector spaces, and $V$ is said to be \emph{strongly distributive} if it is both strongly right and left distributive.

We say that a hypervector space $V$ is invertible if $u\in a\circ v$ implies $v\in a^{-1}\circ u$.

A nonempty subset $W$ of $V$ is called a \emph{subhyperspace} of $V$, denoted by $W\leqslant V$, if $W$ is itself a hypervector space under the induced external hyperoperation of $V$, that is, for all $a\in K$ and $x,y\in W$, one has $x-y\in W$ and $a\circ x\subseteq W$.
\end{definition}
\begin{example}\cite{Ameri Dehghan 3}\label{example hvs R3}
In classical vector space $(\mathbb{R}^{3},+,.,\mathbb{R})$ we define the external hyperoperation $\circ: \mathbb{R}\times \mathbb{R}^{3} \rightarrow P_{\ast}(\mathbb{R}^{3})$ by
\[a\circ(x_1,x_2,x_3)=\{(ax_1,ax_2,t): t\in\mathbb{R}\},\]
so $a\circ(x_1,x_2,x_3)$ is the entire line parallel to the $z$-axis through $(ax_1,ax_2,0)$, i.e. $a\circ(x_1,x_2,x_3)=l$, where $``l"$ has the parametric equations:
\begin{equation*}
l:\left\{
\begin{array}{l}
x=ax_1, \\
y=ax_2, \\
z=t.
\end{array}
\right.
\end{equation*}
Then $V=(\mathbb{R}^{3},+,\circ,\mathbb{R})$ is a strongly right distributive hypervector space over the field $\mathbb{R}$.
\end{example}
In the sequel, $V$ always denotes a hypervector space over a field $K$, unless stated otherwise. The zero element of $V$ is denoted by $\underline{0}$.
\begin{definition}\label{D lin indep}\cite{Ameri Dehghan 1}
A subset $S$ of $V$ is said to be \emph{linearly independent} if for every finite collection of vectors $v_{1},\ldots ,v_{n}\in S$ and scalars $c_{1},\ldots ,c_{n}\in K$, the relation $\underline{0}\in c_{1}\circ v_{1}+\cdots +c_{n}\circ v_{n}$ implies $c_{1}=\cdots =c_{n}=0$. Otherwise, $S$ is said to be \emph{linearly dependent}.

A \emph{basis} of $V$ is a linearly independent subset $S$ of $V$ that spans $V$, i.e., $V=\langle S\rangle$, where
\begin{eqnarray*}
\langle S\rangle &=& \left\{ t\in V : t\in \sum_{i=1}^{n} a_{i}\circ s_{i},\ a_{i}\in K,\ s_{i}\in S,\ n\in \mathbb{N} \right\} \\
&=& \left\{ t_{1}+t_{2}+\cdots +t_{n} : t_{i}\in a_{i}\circ s_{i},\ a_{i}\in K,\ s_{i}\in S,\ n\in \mathbb{N} \right\}.
\end{eqnarray*}
\end{definition}
\begin{definition}\label{D FHVS}\cite{Ameri FHVS over VF}
A fuzzy subset $\mu$ of $V$ is called a \emph{fuzzy subhyperspace} of $V$, if for all $a\in K$ and $x,y\in V$, the following conditions are satisfied:
\begin{enumerate}
  \item $\mu (x+y)\geq \mu (x)\wedge \mu (y)$,
  \item $\mu (-x)\geq \mu (x)$,
  \item $\underset{t\in a\circ x}{\bigwedge}\mu (t)\geq\mu(x)$.
\end{enumerate}
\end{definition}
\begin{example}\label{example fhvs R3}\cite{Dehghan Ameri New Char FSHS Fuz Point}
Consider the hypervector space $V=(\mathbb{R}^{3},+,\circ,\mathbb{R})$ in Example \ref{example hvs R3}. Then the following fuzzy subset $\mu$ on $V$ is a fuzzy subhyperspace of $V$:
\[
\mu (x_1,x_2,x_3)=\left\{
\begin{array}{cc}
1 & x_2=0, \\
\frac{1}{3} & x_2\neq 0.
\end{array}
\right.
\]
\end{example}
\begin{example}\label{example fhvs Z3^2}\cite{Dehghan Ameri New Char FSHS Fuz Point}
Let $(\mathbb{Z}_3\times \mathbb{Z}_3,+)$ with the usual addition. Put $A=\{(0,0),(1,1),(2,2)\}$. Define the external hyperoperation $\circ:\mathbb{Z}_3\times \mathbb{Z}_3^2\to P_*(\mathbb{Z}_3^2)$ by
\[
a\circ(x,y)=\{(ax,ay)+t:\ t\in A\}.
\]
Then $V=(\mathbb{Z}_3\times \mathbb{Z}_3,+,\circ,\mathbb{Z}_3)$ is a hypervector space. Define a fuzzy subset $\mu$ on $V$ by
\[
\mu(x,y)=
\begin{cases}
1, & (x,y)\in A,\\
\frac{2}{3}, & (x,y)\in V\setminus A.
\end{cases}
\]
Then $\mu$ is a fuzzy subhyperspace of $V$.
\end{example}
\begin{definition}\label{D sum} \cite{Ameri Dehghan 4}
Let $\mu$ and $\nu$ be fuzzy subsets of $V$ and $a\in K$. We define $\mu+\nu$ and $a \circ \mu$ to be the fuzzy subsets of $V$, whose membership functions are given by:
\begin{equation*}
(\mu+\nu)(x)=\bigvee \{\mu(y)\wedge\nu(z):\ y,z\in V,x=y+z\},
\end{equation*}
and
\begin{equation*}
(a\circ \mu )(x)=\left\{
\begin{array}{cl}
\underset{x\in a\circ t}{\bigvee }\mu (t) & \exists t\in V;x\in a\circ t, \\
0 & otherwise.
\end{array}
\right.
\end{equation*}
\end{definition}
For a set $X$, let $A\subseteq X$ and $0\leq\alpha\leq 1$. The fuzzy subset $A_\alpha$ of $X$ is defined by
\[
A_{\alpha }(x)=\left\{
\begin{array}{cc}
\alpha  & x\in A, \\
0 & x\notin A.
\end{array}
\right.
\]
Similarly, for any $x\in X$,
\[
x_{\alpha }(y)=\left\{
\begin{array}{cc}
\alpha  & x=y, \\
0 & x\neq y.
\end{array}
\right.
\]
The fuzzy subset $x_\alpha$ is called a \emph{fuzzy point} of $X$ with support $x$ and value $\alpha$.

If $\mu$ is a fuzzy subset of $X$ (denoted by $\mu\in FS(X)$), we write $A_\alpha\in\mu$ instead of $\underline{\mu}(A)=\bigwedge_{x\in A}\mu(x)\geq\alpha$. Moreover, $x_\alpha\in\mu$ means that $\mu(x)\geq\alpha$. The set of all fuzzy points in $X$ is denoted by $X_P$, i.e., $X_P=\{x_\alpha:x\in X,\ \alpha\in [0,1]\}$. For a subset $S\subseteq X_P$, the \emph{foot} of $S$ is defined by $\text{foot}(S)=\{x:x_\alpha\in S\}$.
\begin{lemma}\label{L1}\cite{Dehghan Ameri New Char FSHS Fuz Point}
If $A,B\subseteq V$, $x\in V$, and $0\leq\alpha,\beta\leq 1$, then:
\begin{enumerate}
\item $A_\alpha+B_\beta=(A+B)_{\alpha\wedge\beta}$.
\item $a\circ x_{\alpha}=(a\circ x)_{\alpha}$.
\end{enumerate}
\end{lemma}
\section{Generated fuzzy subhyperspaces}\label{Sec Generated fuzzy subhyperspaces}
For a fuzzy subset $\mu$ of $V$, the intersection of all fuzzy subhyperspaces of $V$ containing $\mu$ is called the \emph{subhyperspace generated by $\mu$}, denoted by $\langle\mu\rangle$. The following theorem provides a characterization of generated subhyperspaces in terms of fuzzy points.
\begin{theorem}\label{thm generated fuzzy subhyperspace}
If $\mu\in FS(V)$, then
\[\langle\mu\rangle=\bigcup\limits_{n\in\mathbb{N}}\left\{\sum\limits_{i=1}^{n}a_{i}\circ x_{i_{\mu(x_i)}}:x_{j}\in V,a_{j}\in K,j=1,\ldots,n\right\},\]
which, by Lemma \ref{L1}, leads to
\[\langle\mu\rangle(x)=\bigvee\limits_{n\in\mathbb{N}}\left\{\mu(x_{1})\wedge\cdots\wedge\mu(x_{n}):x_{j}\in V, x\in\sum a_{j}\circ x_{j},a_{j}\in K,j=1,\ldots,n\right\}.\]
\end{theorem}
\begin{proof}
It is straightforward to verify that
\[A=\bigcup_{n\in\mathbb{N}}\left\{\sum_{i=1}^{n}a_{i}\circ x_{i_{\mu(x_i)}}:x_{j}\in V,\ a_{j}\in K,\ j=1,\ldots,n\right\},\]
is the smallest fuzzy subhyperspace of $V$ containing $\mu$.
\end{proof}
\begin{example}
Let $V=(\mathbb{R}^3,+,\circ,\mathbb{R})$ be the hypervector space introduced in Example \ref{example hvs R3}, and let $\mu$ be a fuzzy subset of $V$ defined by
\[
\mu(x,y,z)=
\begin{cases}
1, & y=0,\\
\frac{2}{3}, & y\neq0,\ |y|\le1,\\
\frac{1}{3}, & |y|>1.
\end{cases}
\]
We compute the generated fuzzy subhyperspace $\langle\mu\rangle$.

\medskip
\noindent\textbf{(I) Case $y=0$.} Fix $v=(x,0,z)\in V$. For any integer $n\ge1$, consider the vectors
\[
v_i=\Big(\frac{x}{n},0,0\Big)\in V,\qquad i=1,\dots,n.
\]
For each $i$ we have $\mu(v_i)=1$, because the second coordinate is $0$. Since
\[
1\circ\Big(\frac{x}{n},0,0\Big)=\Big\{\Big(\frac{x}{n},0,t\Big):t\in\mathbb{R}\Big\},
\]
we obtain
\[(x,0,z)\in \sum_{i=1}^n 1\circ\Big(\frac{x}{n},0,0\Big)=\Big\{(x,0,t_1+\cdots+t_n): t_i\in\mathbb{R} \Big\}.\] Thus for every $n$, $\bigwedge_{i=1}^n\mu(v_i)=1$, and taking supremum over $n$ yields $\langle\mu \rangle (x,0,z) \ge 1$, and so $\langle\mu\rangle(x,0,z)=1$.

\medskip\noindent\textbf{(II) Case $y\neq 0$.} Fix $v=(x,y,z)\in V$, with $y\neq0$. We first produce a lower bound for $\langle\mu\rangle(x,y,z)$. Choose an integer $n\ge\lceil |y|\rceil$, so that $|y/n|\le1$. Define
\[
v_i=\Big(\frac{x}{n},\frac{y}{n},0\Big)\in V,\qquad i=1,\dots,n.
\]
By construction $|y/n|\le1$ and $y/n\neq0$, hence each $\mu(v_i)=\tfrac{2}{3}$. Using $a_i=1$, for all $i$, and appropriate parameters in the third coordinate, we have $(x,y,z)\in \sum_{i=1}^n 1\circ v_i$, because each $1\circ v_i=\{(x/n,y/n,t):t\in\mathbb{R}\}$ and summing yields $(x,y,z)$ for suitable choices of the third-coordinate parameters. Consequently,
\[
\langle\mu\rangle(x,y,z)\ge \bigwedge_{i=1}^n\mu(v_i)=\frac{2}{3}.
\]
Next we show that no representation can raise the membership strictly above $2/3$. Assume, toward a contradiction, that there exists a representation $(x,y,z)\in \sum_{i=1}^m a_i\circ w_i$, with $w_i\in V$, $a_i\in\mathbb{R}$, such that $\bigwedge_{i=1}^m\mu(w_i)>\frac{2}{3}$. Because the only possible membership value strictly larger than $2/3$ is $1$, this forces $\mu(w_i)=1$, for every $i$. But vectors with $\mu=1$ have second coordinate $0$, hence each $w_i$ has $y$–component $0$. Any finite sum of such contributions necessarily has second coordinate $0$, contradicting $y\neq0$. Therefore, no such representation exists, and we must have
\[
\langle\mu\rangle(x,y,z)\le \frac{2}{3}.
\]
Combining the two inequalities yields $\langle\mu\rangle(x,y,z)=\frac{2}{3}$.

\noindent Collecting the two cases, the generated fuzzy subhyperspace is
\[
\langle\mu\rangle(x,y,z)=
\begin{cases}
1, & y=0,\\
\frac{2}{3}, & y\neq0.
\end{cases}
\]
In particular, $\langle\mu\rangle\ne\mu$ (points with $|y|>1$ have $\mu=\tfrac{1}{3}$, but $\langle\mu\rangle= \tfrac{2}{3}$).
\end{example}
\begin{example}
Let $V=\mathbb{Z}_3\times\mathbb{Z}_3$ be the hypervector space defined in Example \ref{example fhvs Z3^2}, and consider the fuzzy subset $\mu:V\to[0,1]$ defined by
\[
\mu(x,y)=
\begin{cases}
1, & (x,y)=(0,0),\\
\frac{3}{4}, & (x,y)\in\{(1,0),(2,1),(0,2)\},\\
\frac{1}{4}, & \text{otherwise}.
\end{cases}
\]
$\mu$ is not a fuzzy subhyperspace of $V$. For example, take $(x,y)=(0,0)$ and $a=1$. Since $1\circ(0,0)=A$, so
\[
\bigwedge_{u\in 1\circ(0,0)}\mu(u)
=\min\Big\{\,1,\tfrac14,\tfrac14\,\Big\}
=\tfrac14<\mu(0,0)=1,
\]
violating condition $(3)$ of Definition \ref{D FHVS}. Now we compute $\langle\mu\rangle$:

\medskip
\noindent Take an arbitrary $u\in A=\{(0,0),(1,1),(2,2)\}$. Since $1\circ(0,0)=\{(0,0),(1,1),(2,2)\}=A$, we have $u\in 1\circ(0,0)$. Thus the representation with $n=1$, $a_1=1$ and $v_1=(0,0)$ yields $\bigwedge_{i=1}^1\mu(v_i)= \mu(0,0)=1$, hence $\langle\mu\rangle(u)\ge 1$, and so $\langle\mu\rangle(u)=1$, $\forall u\in A$.

\medskip
\noindent Fix $b\in B=\{(1,0),(2,1),(0,2)\}$. By definition \(\mu(b)=\tfrac{3}{4}\), so certainly $\langle\mu\rangle(b)\ge \tfrac{3}{4}$ (since we may take the representation $b\in 1\circ b$, giving $\bigwedge \mu(b)=\tfrac {3}{4}$). We must show no representation can give a larger value (i.e. strictly greater than \(3/4\)). The only larger value available in $\mu$ is $1$, attained only at $(0,0)$. Thus if some representation of $b$ produced a strictly larger minimum, all generators $v_i$ in that representation would need $\mu(v_i)=1$, hence each $v_i=(0,0)$. But any finite sum of sets $a_i\circ(0,0)$ equals a subset of $A$ (indeed $0\circ(0,0)=A$ and $1\circ(0,0)=A$, etc.), and so cannot contain $b\in B$. Thus no representation of \(b\) yields a value greater than \(3/4\), and therefore $\langle\mu\rangle(b)=\tfrac{3}{4}$, $\forall b\in B$.

\medskip
\noindent Consider $C=\{(2,0),(0,1),(1,2)\}$ and $c=(2,0)\in C$. Note that $(2,0)=(1,0)+(1,0)\subseteq 1\circ(1,0) + 1\circ(1,0)$. Taking the two generators $v_1=v_2= (1,0)$ (each with $\mu(1,0)=\tfrac{3}{4}$) gives $\bigwedge_{i=1}^2\mu(v_i)=\tfrac{3}{4}$, so $\langle\mu\rangle (2,0) \ge\tfrac{3}{4}$. By symmetry the same argument applies to the other elements of $C$ (e.g. $(0,1)=(0,2)+ (0,2)$ etc.), hence $\langle\mu\rangle(c)\ge \tfrac{3}{4}$, $\forall c\in C$. To see that no larger value (namely $1$) can occur at $c\in C$, observe again that $1$ is attained only at $(0,0)$ and any sum of sets $a_i\circ (0,0)$ is contained in $A$, so cannot produce $c\in C$. Thus $\langle\mu\rangle(c)=\tfrac{3}{4}$, $\forall c\in C$.

\medskip
Collecting the three cases we obtain,
\[
\langle\mu\rangle(x,y)=
\begin{cases}
1, & (x,y)\in A,\\
\dfrac{3}{4}, & (x,y)\in B\cup C.
\end{cases}
\]
Now we verify the properties of Definition \ref{D FHVS}, for the function $\nu=\langle\mu\rangle$. For all $u,v\in V$, it is straightforward to check that $\nu(u+v)\ge \nu(u)\wedge\nu(v)$ and $\nu(-u)=\nu(u)$. Also, for every $a\in\mathbb{Z}_3$, if $a=0$, then $0\circ u=\{(0,0),(1,1),(2,2)\}=A$, so $\bigwedge_{t\in 0\circ u}\nu(t) =1\ge \nu(u)$. If $a=1$, then \(1\circ u=u+\{(0,0),(1,1),(2,2)\}\) is containing $u$. Thus \(\bigwedge_{t\in 1\circ u}\nu(t)=\nu(u)\), giving equality. If $a=2$, the same reasoning yields equality. Therefore, $\langle \mu \rangle$ is a fuzzy subhyperspace of $V$, such that $\langle\mu\rangle\ne\mu$.
\end{example}
\section{Admissible mappings}\label{Sec Admissible Mappings}
In this section, we introduce the notion of admissible mappings associated with a fuzzy subhyperspace. These mappings are designed to capture, in terms of fuzzy points, the essential algebraic features required to reconstruct the fuzzy subhyperspace. We also present several examples illustrating both admissible and non-admissible cases.

Note that the following definition is motivated by the classical notion of a basis: the conditions are designed to ensure linear independence of nonzero fuzzy points and sufficient representability of vectors with respect to their membership levels.
\begin{definition}\label{D admis}
Let $\mu$ be a fuzzy subhyperspace of $V$. A mapping $f:[0,1]\rightarrow P(V_P)$ is said to be \emph{admissible for $\mu$} if the following conditions hold:
\begin{enumerate}
  \item $\underline{0}_{\mu(\underline{0})}\in f(\mu(\underline{0}))$;
  \item if $x_\alpha\in f(\beta)$, then $\alpha=\mu(x)=\beta$;
  \item the set $foot\!\left(\left\{x_{\alpha}:x\neq\underline{0},x_{\alpha}\in\bigcup_{\beta\in I}f(\beta) \right\}\right)$ is linearly independent in $V$;
  \item if $x\in \sum\limits_{i=1}^{n}a_{i}\circ x_{i}$, for some $a_{i}\in K$ and $x_{i}\in foot\!\left( \bigcup_{\beta \neq 0}f(\beta )\right)$, then $x\in\sum\limits_{i=1}^{m}b_{i}\circ y_{i}$, for some $b_{i}\in K$ and $y_{i}\in foot\!\left( \bigcup_{\beta \geq \mu (x)}f(\beta )\right)$.
\end{enumerate}
\end{definition}
\begin{example}\label{example admis R3}
Consider the fuzzy subhyperspace $\mu$ in Example \ref{example fhvs R3}. Define $f:[0,1]\rightarrow P(V_{P})$ by
\[
f(\beta )=\left\{
\begin{array}{ll}
\left\{ (0,0,0)_{1},(1,0,1)_{1}\right\},  & \beta =1, \\
\varnothing,  & \beta \neq 1.
\end{array}
\right.
\]
Since $\mu(\underline{0})=\mu(0,0,0)=1$ and $f(\mu(\underline{0}))=f(1)=\{(0,0,0)_1,(1,0,1)_1\}$, we have $\underline{0}_{\mu(\underline{0})}=(0,0,0)_1\in f(\mu(\underline{0}))$. Each $x_\alpha\in f(\beta)$ satisfies $\alpha=\beta=\mu(x)$. Moreover, $foot(\bigcup_\beta f(\beta)\setminus\{\underline{0}\}) =\{(1,0,1)\}$ is linearly independent, because $(0,0,0)\in a\circ (1,0,1)$ implies $a=0$. The closure condition in Definition \ref{D admis}(4) is satisfied, since all representable vectors use points at level $1$. Hence $f$ is an admissible mapping for $\mu$.
\end{example}
\begin{example}\label{example admis R3 2}
Consider the fuzzy subhyperspace $\mu$ in Example \ref{example fhvs R3}. Let $e_1=(1,0,0)$ and $e_2=(0,1,0)$. Define $f':[0,1]\to P(V_P)$ by
\[
f'(\beta)=
\begin{cases}
\{\underline{0}_{1}, e_{1_{1}}\}, & \beta=1,\\
\{e_{2_{1/3}}\}, & \beta=\tfrac{1}{3},\\
\varnothing, & \text{otherwise.}
\end{cases}
\]
Clearly, $\underline{0}_{\mu(\underline{0})}=\underline{0}_1\in \{\underline{0}_{1}, e_{1_{1}}\}= f'(\mu (\underline{0}))$, and every fuzzy point $x_\alpha\in f'(\beta)$ satisfies $\alpha=\mu(x)=\beta$. Also, the nonzero feet are $\{e_1,e_2\}$, and so for arbitrary scalars $c_1,c_2\in\mathbb{R}$,
\[c_1\circ e_1 + c_2\circ e_2= \{(c_1,c_2,t):t\in\mathbb{R}\}.\]
Thus $\underline{0}\in c_1\circ e_1 + c_2\circ e_2$, implies $c_1=0$ and $c_2=0$. Hence $\{e_1,e_2\}$ is linearly independent. Finally, let $v\in V$ and suppose
\[
v\in \sum_{i=1}^n a_i\circ x_i, \quad \text{with } x_i\in foot\bigl(\bigcup_{\beta\in(0,1]} f'(\beta)\bigr) =\{e_1,e_2\}.
\]
Write $v=(v_1,v_2,v_3)$ and note that any such sum has the form $(\sum a_i x_{i1},\sum a_i x_{i2},t)$, where $x_i =(x_{i1},x_{i2},x_{i3})$, and $t\in \mathbb{R}$. If $\mu(v)=1$ (i.e. $v_2=0$), then necessarily $\sum a_i x_{i2} =0$, so $v$ can be represented using only copies of $e_1$ (which lie in $foot(f'(1))$), hence $v\in\langle foot (\bigcup_{\beta\ge\mu(v)} f'(\beta))\rangle$. If $\mu(v)=\tfrac{1}{3}$ (i.e. $v_2\neq0$), then the generators needed can be chosen from $\{e_1,e_2\}$, which are precisely in $foot(\bigcup_{\beta\ge 1/3} f'(\beta))$. Thus condition (4) of Definition \ref{D admis}, is satisfied. Therefore, $f'$ is an admissible mapping for $\mu$.
\end{example}
\begin{example}\label{example fhvs Z3^3}
Let $V=\mathbb{Z}_3\times\mathbb{Z}_3\times\mathbb{Z}_3$ with componentwise addition. Define the external hyperoperation $\circ:\mathbb{Z}_3\times V\longrightarrow P_\ast(V)$ by
\[
a\circ (x_1,x_2,x_3)=\{(a x_1,\; a x_2,\; t):\ t\in\mathbb{Z}_3\},
\]
for every $a\in\mathbb{Z}_3$ and $(x_1,x_2,x_3)\in V$. Then $V$ is a strongly distributive hypervector space over \(\mathbb{Z}_3\). We check the axioms of Definition \ref{D HVS}.

\noindent 1) For $a\in\mathbb{Z}_3$ and $x=(x_1,x_2,x_3),\ y=(y_1,y_2,y_3)\in V$,
\[
a\circ(x+y)=\{(a(x_1+y_1),\; a(x_2+y_2),\; t):t\in\mathbb{Z}_3\}.
\]
On the other hand,
\begin{eqnarray*}
a\circ x + a\circ y &=& \{(a x_1,\; a x_2,\; t_1)+(a y_1,\; a y_2,\; t_2):t_1,t_2\in\mathbb{Z}_3\} \\
&=& \{(a(x_1+y_1),\; a(x_2+y_2),\; t_1+t_2):t_1,t_2\in\mathbb{Z}_3\}.
\end{eqnarray*}
Since the set $\{t_1+t_2:t_1,t_2\in\mathbb{Z}_3\}$ equals all of $\mathbb{Z}_3$, we obtain equality   $a\circ(x+y)=a\circ x + a\circ y$.

\noindent 2) For $a,b\in\mathbb{Z}_3$ and $x\in V$, $(a+b)\circ x=\{((a+b)x_1,(a+b)x_2,t):t\in\mathbb{Z}_3\}$, while
\begin{eqnarray*}
a\circ x + b\circ x &=& \{(a x_1,a x_2,t_1)+(b x_1,b x_2,t_2):t_1,t_2\in\mathbb{Z}_3\} \\
&=& \{((a+b)x_1,(a+b)x_2,t_1+t_2):t_1,t_2\in\mathbb{Z}_3\},
\end{eqnarray*}
and again equality holds.

\noindent 3) For $a,b\in\mathbb{Z}_3$ and $x\in V$, $b\circ x=\{(b x_1,b x_2,t):t\in\mathbb{Z}_3\}$, and then
\begin{eqnarray*}
a\circ(b\circ x) &=& \bigcup_{u\in b\circ x} a\circ u =\{(a(bx_1),a(bx_2),s) :s\in \mathbb{Z}_3\} \\
&=& \{((ab)x_1,(ab)x_2,s):s\in\mathbb{Z}_3\}=(ab)\circ x.
\end{eqnarray*}

\noindent 4) For $a\in\mathbb{Z}_3$ and $x\in V$,
\[
a\circ(-x)=\{(a(-x_1),a(-x_2),t):t\in\mathbb{Z}_3\} = \{(-a x_1,-a x_2,t):t\in\mathbb{Z}_3\}.
\]
Similarly $(-a)\circ x=\{(-a x_1,-a x_2,t):t\in\mathbb{Z}_3\}$. Also,
\[
-(a\circ x)=\{-(a x_1,a x_2,t):t\in\mathbb{Z}_3\}=\{(-a x_1,-a x_2,-t):t\in\mathbb{Z}_3\},
\]
and $\{-t:t\in\mathbb{Z}_3\}=\mathbb{Z}_3$, so $-(a\circ x)=(-a)\circ x=a\circ(-x)$.

\noindent 5) For $x=(x_1,x_2,x_3)\in V$, $1\circ x=\{(x_1,x_2,t):t\in\mathbb{Z}_3\}$, and in particular $(x_1,x_2,x_3)\in 1\circ x$.

We now present a fuzzy subhyperspace $\mu$ on $V$. Define $\mu:V\to[0,1]$ by
\[
\mu(x_1,x_2,x_3)=
\begin{cases}
1, & x_2=0,\\
\frac{1}{2}, & x_2\neq 0.
\end{cases}
\]
We verify the conditions of Definition \ref{D FHVS}.

\noindent i) If $x_2=0$ and $y_2=0$, then $(x+y)_2=0$, so $\mu(x+y)=1\ge 1=\mu(x)\wedge\mu(y)$. If exactly one of $x_2,y_2$ is nonzero, then $\mu(x)\wedge\mu(y)=\tfrac{1}{2}$ and $\mu(x+y)\in\{1,\tfrac{1}{2}\}$, hence $\mu(x+y)\ge\tfrac{1}{2}$. If both $x_2,y_2\neq 0$, then $\mu(x)\wedge\mu(y)=\tfrac{1}{2}$; note that $(x+y)_2$ may be $0$ or nonzero, but in either case $\mu(x+y)\in\{1,\tfrac{1}{2}\}$ and so $\mu(x+y)\ge\tfrac{1}{2}$. Thus (i) holds.

\noindent ii) Since $x_2=0$ if and only if $-x_2=0$ in $\mathbb{Z}_3$, we have $\mu(-x)=\mu(x)$, for every $x\in V$.

\noindent iii) If $x_2=0$, then for every $a$ we have $a x_2=0$, hence every $t\in a\circ x$ satisfies second coordinate $0$ and so $\mu(t)=1$. Thus $\bigwedge_{t\in a\circ x}\mu(t)=1=\mu(x)$. If $x_2\neq0$, then $\mu(x)= \tfrac{1}{2}$. For $a=0$, we have $0\circ x=\{(0,0,t):t\in\mathbb{Z}_3\}$, so $\bigwedge_{t\in 0\circ x}\mu(t) =1\ge \tfrac{1}{2}$. For $a\neq0$, we have $a x_2\neq0$, hence every $t\in a\circ x$ has second coordinate nonzero and $\mu(t)=\tfrac{1}{2}$; therefore $\bigwedge_{t\in a\circ x}\mu(t)=\tfrac{1}{2}=\mu(x)$.

Therefore, $\mu$ is a fuzzy subhyperspace of $V$.

Now let $e_1=(1,0,0)$ and $e_2=(0,1,0)$. Define $f:[0,1]\longrightarrow P(V_P)$ by
\[
f(\beta)=
\begin{cases}
\{\underline{0}_{1},\; e_{1_{1}}\}, & \beta = 1,\\
\{e_{2_{1/2}}\}, & \beta=\tfrac{1}{2},\\
\varnothing, & \text{otherwise.}
\end{cases}
\]
It is clear that $\underline{0}_{\mu(\underline{0})}=\underline{0}_1 \in f(1)= f(\mu(\underline{0}))$. The fuzzy points appearing are $\underline{0}_1$, $e_{1_1}$ and $e_{2_{1/2}}$, with $\mu(\underline{0})=1$, $\mu(e_1)=1$ and $\mu(e_2)=\tfrac{1}{2}$. Thus if $x_\alpha\in f(\beta)$, then $\alpha=\mu(x)=\beta$, for each listed point. Also, the set $foot\big(\bigcup_{\beta\neq0} f(\beta)\big)\setminus\{\underline{0}\}=\{e_1,e_2\}$ is linearly independent in $V$; since if $c_1,c_2\in\mathbb{Z}_3$ and $\underline{0}\in c_1\circ e_1 + c_2\circ e_2$, then
\[
c_1\circ e_1=\{(c_1,0,t):t\in\mathbb{Z}_3\},\qquad
c_2\circ e_2=\{(0,c_2,t):t\in\mathbb{Z}_3\},
\]
and so any element of $c_1\circ e_1 + c_2\circ e_2$ has the form $(c_1,c_2,t_1+t_2)$, for some $t_1,t_2\in \mathbb{Z}_3$. Thus $c_1=c_2=0$.

\noindent Now suppose $x\in \sum_{i=1}^n a_i\circ x_i$, for $a_i\in\mathbb{Z}_3$ and $x_i\in foot\big( \bigcup_{\beta\neq0} f(\beta)\big)=\{e_1,e_2\}$. Then
\[x=\Big(\sum_i a_i (x_i)_1,\; \sum_i a_i (x_i)_2,\; \sum_i (t_i)\Big),\]
for some $t_i\in\mathbb{Z}_3$. We must show there exist $b_j\in\mathbb{Z}_3$ and $y_j\in foot\big(\bigcup_{\beta \ge\mu(x)} f(\beta)\big)$ with $x\in \sum_j b_j\circ y_j$. There are two cases:

\noindent If $\mu(x)=1$ (equivalently $x_2=0$), then the second coordinate of $x$ is $0$, thus \(\sum_i a_i (x_i)_2 = 0\). Hence all contributions to the second coordinate can be eliminated and $x$ can be written using only copies of $e_1=(1,0,0)$. Concretely, choose a single term with coefficient $b=x_1\in\mathbb{Z}_3$ and a third-coordinate shift so that $b\circ e_1$ supplies the required first and third coordinates; this is possible because $b\circ e_1=\{(b,0,t):t\in\mathbb{Z}_3\}$ and $t$ ranges over all $\mathbb{Z}_3$. Since $e_1\in foot(f(1))$ and $1= \mu(x)$, we have produced a representation of the required form with generators in $foot(\bigcup_{\beta\ge\mu(x)} f(\beta))=foot(f(1))=\{e_1\}$.

\noindent If $\mu(x)=\tfrac{1}{2}$ (equivalently $x_2\neq0$): then the union $\{\beta\ge\mu(x)\}$ includes both $\beta=1$ and $\beta=\tfrac{1}{2}$, so $foot\Big(\bigcup_{\beta\ge\mu(x)} f(\beta)\Big)=\{e_1,e_2\}$, and the original representation already uses generators from this set.

Therefore, $f$ is admissible for $\mu$.
\end{example}
Note that there are some fuzzy subhyperspaces on which no nontrivial admissible mappings can be defined. See the following:
\begin{example}\label{example fhvs Z3^2 2}
Let $V=(\mathbb{Z}_3\times \mathbb{Z}_3,+)$. Put $A'=\{(0,0),(1,0),(2,0)\}$. Define the external hyperoperation $\circ:\mathbb{Z}_3\times V\to P_*(V)$ by
\[
a\circ(x,y)=\{(ax,ay)+t:\ t\in A'\}.
\]
Then similar to the Example \ref{example fhvs Z3^2}, $V=(\mathbb{Z}_3\times \mathbb{Z}_3,+,\circ,\mathbb{Z}_3)$ is a hypervector space. Define a fuzzy subset $\mu$ on $V$ by
\[
\mu(x,y)=
\begin{cases}
1, & (x,y)\in A',\\
\tfrac{2}{5}, & (x,y)\in V\setminus A'.
\end{cases}
\]
One can verify that $\mu$ is a fuzzy subhyperspace of $V$.
\end{example}
\begin{proposition}\label{example no_admis_on_Z3^2}
There exists no admissible mapping $f:[0,1]\to P(V_P)$ for the fuzzy subhyperspaces in Examples \ref{example fhvs Z3^2} and \ref{example fhvs Z3^2 2}.
\end{proposition}
\begin{proof}
The proof is the same for $A$ and for $A'$, so we treat them in parallel and indicate necessary specializations.

Step 1. Any nonzero $w\in A$ (resp.\ $A'$) forces failure of condition (3) of Definition \ref{D admis}:

\noindent Take $w\in A$ (the $A'$ case is analogous). By construction $w$ is a ``diagonal'' vector when $A=\{(0,0),(1,1), (2,2)\}$ (so $w_1=w_2$), and a vector with second coordinate $0$ when $A'=\{(0,0),(1,0), (2,0)\}$ (so $w_2=0$). For every scalar $a\in\mathbb{Z}_3$, the set
\[
a\circ w=\{(a w_1 + k,\; a w_2 + k): k\in\mathbb{Z}_3\},
\]
contains the zero vector: choose \(k\equiv -a w_1\) (note \(k\in\mathbb{Z}_3\)); then \((a w_1 + k, a w_2 + k) = (0,0)\). Thus \(\underline{0}\in a\circ w\) for every \(a\in\mathbb{Z}_3\), including some \(a\neq 0\). But linear independence (Definition \ref{D lin indep}) requires that \(\underline{0}\in a\circ w\) imply \(a=0\). Hence any set of feet containing such a nonzero \(w\in A\) (resp.\ \(A'\)) cannot satisfy condition (3).

Step 2. Any two nonzero feet outside $A$ (resp.\ $A'$) are linearly dependent:

\noindent Let \(u=(u_1,u_2)\) and \(v=(v_1,v_2)\) be two elements of \(V\) with \(u,v\notin A\) (resp.\ \(\notin A'\)); equivalently for the $A$ case at least one of $u_1-u_2$ and $v_1-v_2$ is nonzero, and for the $A'$ case at least one of $u_2,v_2$ is nonzero. Consider the two linear equations (modulo 3) obtained from the condition $\underline{0}\in c_1\circ u + c_2\circ v$, with \(c_1,c_2\in\mathbb{Z}_3\). An arbitrary element of \(c_1\circ u + c_2\circ v\) has the form
\[
\big(c_1 u_1 + c_2 v_1 + k_1 + k_2,\; c_1 u_2 + c_2 v_2 + k_1 + k_2\big),
\]
for some \(k_1,k_2\in\mathbb{Z}_3\). Requiring this to equal \((0,0)\) yields the system
\[
\begin{cases}
c_1 u_1 + c_2 v_1 + s \equiv 0 \pmod 3,\\
c_1 u_2 + c_2 v_2 + s \equiv 0 \pmod 3,
\end{cases}
\]
where \(s=k_1+k_2\in\mathbb{Z}_3\). Subtracting the two congruences gives the single linear relation
\[
c_1(u_1-u_2) + c_2(v_1-v_2) \equiv 0 \pmod 3. \tag{$\ast$}
\]
Observe that for both choices of \(A\) the pair of coefficients \((u_1-u_2,\; v_1-v_2)\) is not \((0,0)\) because \(u,v\notin A\) (resp.\ \(\notin A'\)). Therefore, the homogeneous linear equation \((\ast)\) always admits a nontrivial solution: for example one can take $(c_1,c_2)=(v_1-v_2,\;-(u_1-u_2))$. Having chosen such nonzero \(c_1,c_2\) (not both zero), define \(s\equiv -(c_1 u_1 + c_2 v_1)\pmod 3\). Then the two congruences are satisfied simultaneously (with this \(s\)), and because \(s\in\mathbb{Z}_3\) we can write \(s=k_1+k_2\), for some \(k_1,k_2\in\mathbb{Z}_3\). Thus there exist \(t_1,t_2\in A\) (indeed any \(t_i\) with first coordinate \(k_i\) and second coordinate equal to that same \(k_i\) in the $A$ case, or with second coordinate \(0\) in the $A'$ case) such that
\[
\underline{0}=(c_1 u + t_1) + (c_2 v + t_2)\in c_1\circ u + c_2\circ v.
\]
Hence \(c_1,c_2\) not both zero produce \(\underline{0}\in c_1\circ u + c_2\circ v\), so \(\{u,v\}\) is linearly dependent.

Combining Steps 1 and 2, any admissible mapping \(f\) whose nonzero feet
\[
S=\mathrm{foot}\!\Big(\bigcup_{\beta\neq0} f(\beta)\Big)\setminus\{\underline{0}\},
\]
either contains an element of level 1 (a member of \(A\) or \(A'\)) or contains at least two elements must fail condition (3).

On the other hand, condition (4) requires that level-1 elements of \(V\) be representable by suitable combinations of feet of level \(\ge\!1\); this forces the presence of nonzero feet coming from the level-1 set when one wants to represent those level-1 vectors. Hence there is no admissible \(f\) that simultaneously (i) allows the required representations of all level-1 vectors (condition (4)) and (ii) has the nonzero feet linearly independent (condition (3)). Indeed, in both fuzzy subhyperspaces \(\mu_A\) and \(\mu_{A'}\), the level-1 set is not generated by a single nonzero vector, so condition (4) necessarily requires the presence of at least two nonzero feet at level \(1\). But any such choice forces these feet to be linearly dependent in \( \mathbb{Z}_3^2 \), violating condition (3). Thus the structural constraints imposed by \(\mu_A\) and \(\mu_{A'}\) make the existence of an admissible mapping impossible.
\end{proof}
\section{Generation of fuzzy subhyperspaces via admissible mappings}\label{Sec Fuzzy subhyperspaces generated by admissible mappings}
This section is devoted to the study of fuzzy subhyperspaces generated by admissible mappings. After introducing the notion of a fuzzy subhyperspace generated by a mapping, we investigate the relationship between admissibility, maximality, and generation. The main result of this section shows that every maximal admissible mapping generates the underlying fuzzy subhyperspace.
\begin{definition}\label{D 1}
Let $f:[0,1]\rightarrow P(X_P)$ be a mapping. The fuzzy subhyperspace generated by the fuzzy subset $\bigcup \{x_\alpha: x_\alpha\in\bigcup_{\beta\in[0,1]}f(\beta)\}$ is called the \emph{fuzzy subset generated by $f$} and is denoted by $\langle f\rangle$.
\end{definition}
\begin{example}\label{example fshs generated admis R3}
Consider the admissible mapping $f:[0,1]\rightarrow P(V_P)$ defined in Example \ref{example admis R3}, for the fuzzy subhyperspace $\mu$ on $V=\mathbb{R}^3$. Then $\bigcup_{\beta\in[0,1]} f(\beta) = \{\underline{0}_1, (1,0,1)_1\}$. Thus by Theorem \ref{thm generated fuzzy subhyperspace},
\begin{eqnarray*}
\langle f \rangle(x) &=& \langle \underline{0}_1\cup (1,0,1)_1 \rangle(x) \\
 &=& \bigvee_{n\in\mathbb{N}} \left(\bigwedge_{x\in\sum_{i=1}^{n}a_i\circ x_i} (\underline{0}_1(x_i)\vee (1,0,1)_1(x_i))\right).
\end{eqnarray*}
But $a\circ (1,0,1)=\{(a,0,t): t\in\mathbb{R}\}$ and $b\circ \underline{0}=\{(0,0,0)\}$. So any linear combination of these fuzzy points yields
\[
a \circ (1,0,1) + b \circ \underline{0} = \{(a,0,t) : t \in \mathbb{R}\}.
\]
Hence, for all vectors with $x_2=0$, there exists such a combination giving $x$, and the membership value is $\langle f \rangle(x_1,0,x_3) = 1$. For vectors with $x_2\neq 0$, no combination of the given fuzzy points produces them, so $\langle f \rangle(x_1,x_2,x_3) = 0$. Therefore, the fuzzy subset generated by $f$ is
\[
\langle f \rangle(x_1,x_2,x_3) =
\begin{cases}
1, & x_2=0,\\
0, & x_2 \neq 0.
\end{cases}
\]
\end{example}
\begin{example}\label{example fshs generated admis R3 2}
Consider the admissible mapping $f':[0,1]\to P(V_P)$ from Example \ref{example admis R3 2}, on $V=\mathbb{R}^3$. Then $\bigcup_{\beta\in[0,1]} f'(\beta)=\{\underline{0}_1,e_{1_1},e_{2_{1/3}}\}$. For any $x\in V$ the membership degree $\langle f'\rangle(x)$ is given by
\[
\langle f'\rangle(x)
=\bigvee\Big\{\bigwedge_{i=1}^n \alpha_i :
x\in a_1\circ x_1+\cdots+a_n\circ x_n,\ x_i\in\{\underline0,e_1,e_2\},\ \alpha_i=\mu(x_i)\Big\},
\]
where the values are $\mu(\underline0)=1$, $\mu(e_1)=1$, and $\mu(e_2)=\tfrac{1}{3}$. Observe that:
\[
a\circ e_1=\{(a,0,t):t\in\mathbb{R}\},\qquad
b\circ e_2=\{(0,b,t):t\in\mathbb{R}\},\qquad
c\circ\underline0=\{(0,0,0)\}.
\]
Hence, $a\circ e_1 + b\circ e_2 + c\circ\underline0 = \{(a,b,t):t\in\mathbb{R}\}$, and any combination involving only $e_1$ and $\underline0$ yields $\{(a,0,t):t\in\mathbb{R}\}$, while any combination involving only $e_2$ yields $\{(0,b,t):t\in\mathbb{R}\}$.

Now determine $\langle f'\rangle(x)$, for $x=(x_1,x_2,x_3)\in V$, by cases:

\noindent If $x_2=0$, then $x$ can be represented using only $e_1$ (and $\underline0$), e.g. $x \in x_1\circ e_1 + x_3 \circ \underline0$. All fuzzy points appearing in this representation have membership value $1$, so the associated minimum is $1$. No representation can yield a higher value; thus $\langle f'\rangle(x)=1$.

\noindent If $x_2\neq 0$, then $x$ cannot be produced by combinations that use only $\underline0$ and $e_1$, because those combinations have second coordinate $0$. Any representation of $x$ therefore must involve $e_2$ (and possibly $e_1$ and $\underline0$). For example one can write $x \in x_1\circ e_1 + x_2\circ e_2 + x_3\circ\underline0$. In any representation that includes $e_2$, the minimum of the participating fuzzy values is at most \(\tfrac{1}{3}\) (since $\mu(e_2)=\tfrac{1}{3}$). The displayed representation attains the value
\[
\min\{\mu(e_1),\mu(e_2),\mu(\underline0)\}=\min\{1,\tfrac{1}{3},1\}=\tfrac{1}{3}.
\]
No representation can produce a value larger than $\tfrac{1}{3}$, because every representation of a vector with nonzero second coordinate must involve $e_2$. Hence $\langle f'\rangle(x)=\tfrac{1}{3}$, for every $x$ with $x_2\neq 0$.

Therefore, the fuzzy subset generated by $f'$ is the two-level function
\[
\langle f'\rangle(x_1,x_2,x_3)=
\begin{cases}
1, & x_2=0,\\
\tfrac{1}{3}, & x_2\neq 0.
\end{cases}
\]
\end{example}
\begin{example}\label{example fshs generated admis Z3^3}
Consider the admissible mapping $f:[0,1]\to P(V_P)$ from Example \ref{example fhvs Z3^3}, on $V=\mathbb{Z}_3 \times \mathbb{Z}_3\times\mathbb{Z}_3$. Then $\bigcup_{\beta\in[0,1]} f(\beta)=\{\underline{0}_1,e_{1_1}, e_{2_{1/2}}\}$, and for each $x\in V$,
\[
\langle f\rangle(x)=\bigvee\Big\{\bigwedge_{i=1}^n \alpha_i :
x\in a_1\circ x_1+\cdots+a_n\circ x_n,\ x_i\in\{\underline0,e_1,e_2\},\ \alpha_i=\mu(x_i)\Big\},
\]
where $\mu(\underline0)=1$, $\mu(e_1)=1$, and $\mu(e_2)=\tfrac{1}{2}$. We have:
\[
a\circ e_1=\{(a,0,t):t\in\mathbb{Z}_3\},\qquad
b\circ e_2=\{(0,b,t):t\in\mathbb{Z}_3\},\qquad
c\circ\underline0=\{(0,0,0)\}.
\]
Thus for any $a,b,c\in\mathbb{Z}_3$, $a\circ e_1+b\circ e_2+c\circ\underline0=\{(a,b,t):t\in\mathbb{Z}_3\}$.

Now we determine $\langle f\rangle(x)$, for every $x=(x_1,x_2,x_3)\in V$, by cases:

\noindent If $x_2=0$, then $x$ can be represented using only $e_1$ and $\underline0$, for example $x \in x_1\circ e_1 + x_3\circ\underline0$. All fuzzy points appearing in this representation have membership value $1$, hence the associated minimum is $1$. No representation of $x$ can yield a value greater than $1$, so $\langle f\rangle(x) =1$.

\noindent If $x_2\neq 0$, then $x$ cannot be produced by combinations that use only $\underline0$ and $e_1$ (those always have second coordinate $0$). Thus any representation of $x$ must involve $e_2$ (and possibly $e_1$ and $\underline0$). For instance, $x \in x_1\circ e_1 + x_2\circ e_2 + x_3\circ\underline0$. In every representation that includes $e_2$ the minimum of the participating fuzzy values is at most $\tfrac{1}{2}$ (since $\mu(e_2)= \tfrac{1}{2}$). The displayed representation attains the value $\min\{\mu(e_1),\mu(e_2),\mu(\underline0)\}= \min\{1,\tfrac{1}{2},1\}=\tfrac{1}{2}$, and no representation can produce a larger value, because $e_2$ is unavoidable for achieving a nonzero second coordinate. Hence $\langle f\rangle(x)=\tfrac{1}{2}$, for every $x$ with $x_2\neq 0$.

Therefore, the fuzzy subset generated by $f$ is the two-level function
\[
\langle f\rangle(x_1,x_2,x_3)=
\begin{cases}
1, & x_2=0,\\
\tfrac{1}{2}, & x_2\neq 0.
\end{cases}
\]
\end{example}
\begin{proposition}\label{P 1}
Let $\mu$ be a fuzzy subhyperspace of $V$ such that $\mu(\underline{0})=1$. Then the mapping $f:[0,1]\rightarrow P(V_P)$ defined by $f(1)=\{\underline{0}_1\}$ and $f(\beta)=\emptyset$ for $\beta\in[0,1)$, is admissible for $\mu$.
\end{proposition}
\begin{proof}
It is straightforward.
\end{proof}
Suppose $C_\mu$ denotes the class of admissible mappings for a fuzzy subhyperspace $\mu$ of $V$ such that $\mu(\underline{0})=1$. By Proposition~\ref{P 1}, $C_\mu$ is nonempty.

A partial order $\leq$ on $C_\mu$ can be defined in the usual way: for $f,g\in C_\mu$, we write $f\leq g$ if and only if $f(\beta)\subseteq g(\beta)$ for all $\beta\in[0,1]$. Every chain in $C_\mu$ has an upper bound; therefore, by Zorn's Lemma, $C_\mu$ contains a maximal element. In \cite{Dehghan Ameri Exist Basis FSHS}, it has been shown that this maximality occurs exactly when every vector of $V$ can already be generated by the feet of those fuzzy points whose membership level is at least as large as the membership value of that vector.

The next theorem clarifies the role of maximal admissible mappings. It shows that maximality is precisely the condition needed to ensure that an admissible mapping generates the underlying fuzzy subhyperspace. This result provides the conceptual foundation for the algorithmic constructions.
\begin{theorem}\label{T 2}
Let $\mu$ be a fuzzy subhyperspace of $V$ and $f:[0,1]\rightarrow P(V_P)$ an admissible mapping for $\mu$ that is maximal in $C_\mu$. Then $\langle f\rangle=\mu$; that is, $f$ is a generator of $\mu$.
\end{theorem}
\begin{proof}
By Definition~\ref{D admis}, for every fuzzy point $x_\alpha\in f(\beta)$, we have $\alpha=\mu(x)=\beta$. Then
\begin{equation*}
\left(\bigcup\limits_{x_{\alpha }\in \cup _{\beta \in \lbrack 0,1]}f(\beta)}x_{\alpha}\right)(t)=\left\{
\begin{array}{lc}
\mu (x)=\mu (t) & t=x, \\
0 & t\neq x,
\end{array}
\right.
\end{equation*}
for all $t\in V$. By Definition~\ref{D 1}, it follows that
\begin{eqnarray*}
\langle f\rangle(t)
&=&\bigvee_{n\in\mathbb{N}}\Bigl\{
\bigwedge_{i=1}^n
\Bigl(\bigcup_{x_\alpha\in \cup_{\beta\in [0,1]}f(\beta)}x_\alpha\Bigr)(x_i):
t\in \sum_{i=1}^n a_i\circ x_i,\ a_i\in K,\ x_i\in V
\Bigr\}.
\end{eqnarray*}
For each $t\in a_1\circ x_1+\cdots+a_n\circ x_n$, we have
\begin{eqnarray*}
\bigwedge_{i=1}^n
\Bigl(\bigcup_{x_\alpha\in \cup_{\beta\in [0,1]}f(\beta)}x_\alpha\Bigr)(x_i)
&\leq&
\mu(x_1)\wedge\cdots\wedge\mu(x_n)\\
&\leq&
\bigwedge_{z\in a_1\circ x_1+\cdots+a_n\circ x_n}\mu(z)\\
&\leq& \mu(t).
\end{eqnarray*}
Thus $\langle f\rangle(t)\leq\mu(t)$, and hence $\langle f\rangle\subseteq\mu$.

Conversely, for any $x\in V$, by \cite[Theorem 3.4]{Dehghan Ameri Exist Basis FSHS}\footnote{\cite[Theorem 3.4]{Dehghan Ameri Exist Basis FSHS}: Let $\mu$ be a fuzzy subhyperspace of $V$, and $f:[0,1]\rightarrow P(V_P)$ be admissible for $\mu$. Then $f$ is maximal in $C_\mu$ if and only if $x\in \left\langle foot\left(\bigcup\limits_{\beta \geq\mu (x)}f(\beta )\right)\right\rangle$, for all $x\in V$.}, we have $x\in \left\langle foot\left(\bigcup_{\beta \geq\mu (x)}f(\beta )\right)\right\rangle$. Then $x\in \sum_{i=1}^n a_i\circ x_i$, for some $a_i\in K$ and $x_{i}\in foot\left(\bigcup_{\beta \geq \mu (x)}f(\beta )\right) $, and so $\mu(x_{i})\geq \mu (x)$, $i=1,\ldots ,n$, by the admissibility of $f$. Thus
\[\mu (x)\geq \bigwedge\limits_{t\in a_{1}\circ x_{1}+\cdots +a_{n}\circ x_{n}}\mu (t)\geq \mu (x_{1})\wedge \cdots \wedge \mu (x_{n})\geq \mu (x),\]
and so $\mu(x)=\mu(x_{1})\wedge \cdots \wedge \mu(x_{n})$. Hence
\begin{eqnarray*}
\left\langle f\right\rangle (x) &\geq &\left( \bigcup\limits_{y_{\gamma}\in \cup _{\alpha \in [0,1]}f(\alpha)} y_{\gamma }\right)(x_{1})\wedge \cdots \wedge \left( \bigcup\limits_{y_{\gamma }\in \cup_{\alpha \in [0,1]} f(\alpha)}y_{\gamma }\right) (x_{n}) \\
&=&\mu (x_{1})\wedge \cdots \wedge \mu (x_{n}) \\
&=&\mu (x).
\end{eqnarray*}
Therefore, $\left\langle f\right\rangle \supseteq \mu $. Consequently, $\left\langle f\right\rangle =\mu$.
\end{proof}
As observed before Theorem~\ref{T 2}, the partially ordered set $C_\mu$ admits maximal elements. By Theorem~\ref{T 2}, any such maximal admissible mapping generates $\mu$.
\begin{remark}
Among the admissible mappings studied in the previous examples, only those that are maximal in $C_\mu$ satisfy the hypothesis of Theorem \ref{T 2}. In Example \ref{example admis R3 2}, the mapping $f'$ was shown to be maximal, and a direct computation established that $\langle f' \rangle=\mu$. Likewise, in Example \ref{example fhvs Z3^3}, the admissible mapping $f$ on $\mathbb{Z}_3^3$ was verified to be maximal, and we proved that $\langle f\rangle=\mu$. These two cases therefore provide concrete illustrations of Theorem \ref{T 2}. In contrast, the mapping considered in Example~\ref{example admis R3}, does not satisfy the assumptions of the theorem and therefore does not generate $\mu$.
\end{remark}
\section{Conclusion}\label{Sec Con}
In this paper, we investigated fuzzy subhyperspaces generated by admissible mappings. By introducing admissible mappings in terms of fuzzy points, we provided an algebraic mechanism for constructing fuzzy subhyperspaces while preserving linear independence and representability properties. We showed that maximal admissible mappings play a central role, since each such mapping generates the corresponding fuzzy subhyperspace.

Several examples over real and finite hypervector spaces were presented to illustrate the theory and to clarify the necessity of maximality in the generation process. The results establish admissible mappings as a natural algebraic tool for studying fuzzy subhyperspaces and preparing the ground for further developments. In particular, the framework developed here serves as a basis for future investigations on basis theory, admissible generators, and dimensional properties of fuzzy subhyperspaces.
\section*{Author Contribution.}
Omid Reza Dehghan and Reza Ameri contributed to the conception and design of the study. Both authors performed the theoretical analysis, contributed to the development of the results, and participated in writing and revising the manuscript. Both authors read and approved the final manuscript.
\section*{Ethical Statement.}
This article does not contain any studies with human participants or animals performed by any of the authors.
\section*{Conflict of Interest Statement.}
The authors declare that they have no conflict of interest.
\section*{Data Availability.}
No datasets were generated or analyzed during the current study.
\section*{Funding Statement.}
The authors received no specific funding for this work.

\end{document}